\documentclass[12pt,a4paper,twoside,english,american]{scrartcl}
\usepackage{helvet}
\usepackage{url}
\usepackage[T1]{fontenc}
\usepackage[latin9]{inputenc}
\usepackage{amsmath}
\usepackage{setspace}
\usepackage{amssymb}
\usepackage{mathtools}
\usepackage{enumerate}

\makeatletter

 \usepackage{amsthm}
 \theoremstyle{plain}

  \theoremstyle{plain}
  \newtheorem*{thm*}{Theorem}
  \theoremstyle{plain}
  
  \theoremstyle{remark}
  
  \theoremstyle{definition}
  
 \theoremstyle{definition}
 \newtheorem*{defn*}{Definition}
  \theoremstyle{plain}
  
 \theoremstyle{definition}
  
  \theoremstyle{remark}
\usepackage[T1]{fontenc}    % Silbentrennung auch nach Umlauten m?glich

\usepackage{a4wide}        % A4-Format
\usepackage{tu-preprint}
\usepackage{amsmath}

\usepackage{euscript}
\usepackage{mathabx}
\allowdisplaybreaks[1] % allow page breaks in some displayed equations
\usepackage{amsfonts}
\newenvironment{keywords}{ \noindent\footnotesize\textbf{Keywords and phrases:}}{}

\newenvironment{class}{\noindent\footnotesize\textbf{Mathematics subject classification 2010:}}{}

\usepackage{color,tu-preprint}

\newcommand{\Abs}[1]{\left\lVert#1\right\rVert}
\newcommand{\R}{\mathbb{R}}

\newcommand*{\abs}[1]{\lvert#1\rvert}

\newcommand{\s}[1]{\mathcal{#1}}
\DeclareMathAccent{\Circ}{\mathalpha}{operators}{"17}

\theoremstyle{plain}% default
\newtheorem{Sa}[subsection]{Theorem}
\newtheorem{Sa*}[section]{Theorem}
\newtheorem{Le}[subsection]{Lemma}
\newtheorem{Le*}[section]{Lemma}
\newtheorem{Fo}[subsection]{Corollary}
\newtheorem{Fo*}[subsection]{Corollary}

\newtheorem{Prop*}[section]{Proposition}

\theoremstyle{definition}
\newtheorem*{Def}{Definition}

\newtheorem{Bei}[subsection]{Example}

\theoremstyle{remark}
\newtheorem{Bem}[subsection]{Remark}

 \numberwithin{equation}{section}

\DeclareMathOperator{\lin}{lin}

\DeclareMathOperator{\range}{range}

\newcommand{\N}{\mathbb{N}}

\newcommand{\eps}{\varepsilon}

\DeclareMathOperator{\1}{\chi}

\newcommand{\ben}{\begin{enumerate}[(i)]}
\newcommand{\een}{\end{enumerate}}

\renewcommand*{\epsilon}{\varepsilon}
\renewcommand*{\rho}{\varrho}

\makeatother

\usepackage{babel}

\begin{document}
\selectlanguage{english}%
\institut{Institut f\"ur Analysis}

\preprintnumber{MATH-AN-06-2013}

\preprinttitle{A note on causality in reflexive Banach spaces.}

\author{Marcus Waurick} 

% \makepreprinttitlepage

\selectlanguage{american}%
\setcounter{section}{-1}

\date{}

\title{A note on causality in Banach spaces.}

\author{Marcus Waurick\\
Institut f\"ur Analysis, Fachrichtung Mathematik\\
Technische Universit\"at Dresden\\
Germany\\
marcus.waurick@tu-dresden.de }
\maketitle
\begin{abstract} In this note we provide examples that show that a common notion of causality for linear operators on Banach spaces does not carry over to the closure of the respective operators. We provide an alternative definition for causality, which is equivalent to the usual definition for closed linear operators but does carry over to the closure.
\end{abstract}
\begin{keywords} Causality, Evolutionary Equations, Resolution space
\end{keywords}

\begin{class} 46B99, 46N99, 47A99
\end{class}

%\newpage

%\tableofcontents{}

\section{Introduction}

In physical processes there is a natural direction of time. This direction may be characterized by causality. When describing physical processes by means of mathematical models one thus needs a definition of this concept in mathematical terms. There are plenty of such in the literature, see e.g.~\cite{Saeks1970} and the references therein, see also \cite{Pagl2013} for causality concepts in the computer sciences and \cite{Jac2000} for a discrete-time analogue of causality. We start out with the definition of causality given in \cite{Saeks1970}, which can be understood as a (common) generalization of the concepts in \cite{Georgiou1993,Jac2000,Kal2013,Par2004,Picard,Trostorff2013,Thomas1997,Weiss2003} and, in a Banach space setting, \cite{Laksh10}. 

We note here that in particular situations, there are several characterizations or sufficient criteria of causal (and time-translation-invariant) linear operators at hand, see e.g.~\cite{Jac2000,Par2004,Picard,Trostorff2013}. In \cite{Jac2000,Thomas1997} and \cite[pp 49]{Par2004} the structure of time-translation invariant operators is exploited with the help of the $z$-transform (for a discrete-time setting) and the Laplace transform (for a continuous-time setting).

 In \cite[Example 6]{Jac2000} it has already been observed that the concepts of causality mentioned have the drawback that for (possibly unbounded) closable operators the operator itself may be causal, whereas its closure is not. In \cite[Example 6]{Jac2000} an example in a discrete-time setting is given, see also Example \ref{bei:time-shift2} below for an example in continuous-time.  We will present a possible definition of (norm-)\emph{strong causality} relying on a certain continuity property, which, for closed operators on reflexive Banach spaces, coincides with the usual notion of causality (Theorem \ref{thm:causal_cont_causal}), and which is stable under closure procedures (Lemma \ref{le:strongly_causal_is_equiv_strong_closure}). We shall note here that the latter issue was also adressed in \cite[Section 6]{Jac2000}. However, in \cite{Jac2000} the authors focus on the time-translation invariant case, which we will not assume in our considerations, see in particular \cite[Section 7]{Jac2000} for the continuous-time case. In Section \ref{sec:nonrefl}, we give a possible generalization to the non-reflexive setting (Theorem \ref{thm:caus_non-refl}). However, we have to restrict ourselves to the densely defined, continuous operator case. Moreover, we shall note here that the characterization for a linear, densely defined $M$ is rather technical, which may result in limited applicability. Therefore, the result should be read in the way that it is possible to define causality in terms of  continuity of a certain mapping \emph{independently} of the chosen core for $M$.

 We also mention that in \cite[Section 3]{WaurickIDAEP2013} we have used the notion of (norm-)strong causality in the analysis of solution operators of certain integro-differential-algebraic evolutionary problems of mathematical physics in the reflexive Banach space setting.

\section{The reflexive case}

We introduce the concept of a resolution space. 

\begin{Def}[{\cite{Saeks1970}}] Let $X$ be a Banach space, $(P_t)_{t\in \R}$ in $L(X)$ a \emph{resolution of the identity}, i.e., for all $t\in \R$ the operator $P_t$ is a projection, $\range(P_t)\subseteqq \range(P_s)$ if and only if $t\leqq s$ and $(P_t)_t$ converges in the weak operator topology to $0$ and $1$ if $t\to -\infty$ and $t\to\infty$, respectively. The pair $(X,(P_t)_t)$ is called \emph{resolution space}. 
\end{Def}

We remark here that the properties of the resolution are only to model the notion of causality. In fact, in the definition of causality, the only thing needed is that $P_t$ are continuous projections for all $t\in \R$. Moreover, we also do not need to assume that $(P_t)_t$ is directed in the above sense. We comment on this issue below. A particular instance of the resolution space is the following.

\begin{Bei} (a) Let $X\coloneqq L^2(\R)$ and let $P_t$ be given by multiplication with the cut-off function $\1_{\R_{< t}}$, $t\in\R$, i.e.,  $P_tf(x)= \1_{\R_{<t}}(x)f(x)$ for a.e.\ $x\in\R$ and $f\in L^2(\R)$.  Then $(X,(P_t)_t)$ is a resolution space, we call $(X,(P_t)_t)$ \emph{standard resolution (s.r.)}

(b) Let $X\coloneqq \ell^2(\N)$ and let $P_t$ be given by $(P_t(x_n)_n)_k \coloneqq x_k$ if $k\geqq t $ and $(P_t(x_n)_n)_k\coloneqq 0$ if $k<  t$, for all $(x_n)_n \in \ell^2(\N)$, $k\in \N$, $t\in\R$. In this way it is possible to treat the discrete-time case. 
\end{Bei}

Let us recall the concept of causality. 

 \begin{Def}[Causality, {\cite{Saeks1970}}] Let $(X,(P_t)_t)$ be a resolution space, $M\colon D(M)\subseteqq X\to X$. We say that $M$ is \emph{causal} (with respect to $(P_t)_t$) if for all $a\in \R$ and $f,g\in D(M)$ with $P_a(f-g)=0$ we have $P_a(M(f)-M(g))=0$. 
 \end{Def}

\begin{Bem}\label{rem:sec1,linm}
(a) If $M$ is linear, then $M$ is causal if and only if $P_af=0$ implies $P_aMf=0$ for all $a\in\R$.

(b) Under certain constraints on the domain of $M$, we can reformulate causality as follows. This characterizing property is given as a definition of causality in several works, see e.g.~\cite{Picard,Weiss2003,Thomas1997}. Assume that $P_a[D(M)]\subseteqq D(M)$ for all $a\in\R$. Then $M$ is causal if and only if for all $a\in \R$ we have $P_aM = P_aMP_a$. Indeed, if $M$ is causal, let $f\in D(M)$, $a\in\R$ and define $g\coloneqq P_af\in D(M)$. Then, obviously, $P_a(f-g)=0$. By causality we deduce that 
\[
 P_aM(P_af) = P_aM(g)=P_aM(f).
\]For the converse, let $f,g\in D(M)$, $a\in \R$ with $P_a(f-g)=0$. Then we get that
\[
  P_a \left( M(f)-M(g)\right)= P_a\left( M(P_af)-M(P_ag)\right)=0.\qedhere
\]
(c) Assume that $M$ is uniformly continuous and that for all $a\in \R$, we have $D(MP_a)\cap D(M)$ is dense in $D(\overline{M})$, where $\overline{M}$ denotes the (well-defined, uniformly) continuous extension of $M$. Then $P_aM P_a = P_aM$ on $D(MP_a)\cap D(M)$ for all $a\in \R$ implies causality for $\overline{M}$. Indeed, it suffices to observe that both $P_aM P_a$ and $P_aM$ are uniformly continuous.
\end{Bem}

\begin{Bei}\label{bei:time-shift} For $h\in\R$ we define $\tau_h\colon C_c(\R)\subseteqq L^2(\R)\to L^2(\R), f\mapsto f(\cdot+h)$. Then it is easy to see that $\tau_h$ is causal with respect to the s.r. if and only if $h\leqq 0$. 
\end{Bei}

If $M$ is assumed to be closable, the definition of causality does \emph{not} carry over to the closure of $M$. The following example illustrates this fact.

\begin{Bei}\label{bei:time-shift2} Consider the s.r. $(L^2(\R),(P_t)_t)$. Let $\s H\coloneqq \lin\{ x\mapsto x^n e^{-\frac{x^2}{2}}; n\in\N_0\}$ be the linear span of all Hermite functions. Now, $\s H$ is dense in $L^2(\R)$, see e.g.~\cite{Wer2007}. Moreover, for any two elements $\gamma_1,\gamma_2\in \s H$ the equality $P_a\gamma_1 = P_a\gamma_2$ for some $a\in\R$ implies $\gamma_1=\gamma_2$. In consequence, every mapping $M:\s H\subseteqq L^2(\R)\to L^2(\R)$ is causal with respect to the s.r. In particular, the shift $\tau_h$ as introduced in Example \ref{bei:time-shift} defined on $\mathcal{H}$ is closable and it is causal even for $h>0$. 
\end{Bei}

As we have seen above the notion of causality is a certain compatibility notion for projections given by a resolution of identity. In order to streamline the proofs, we introduce the concept of compatibility at first.

\begin{Def} Let $X$ be a Banach space, $P\in L(X)$, $M\colon D(M)\subseteqq X\to X$. We call $M$ \emph{$P$-compatible} if for all $f,g\in D(M)$ we have $PM(f)=PM(g)$ provided that $Pf=Pg$.  
\end{Def}

\begin{Bem}\label{rem:pcomp_pcausal} Observe that if $(X,(P_t)_t)$ is a resolution space and $M\colon D(M)\subseteqq X\to X$ then $M$ is causal w.r.t. $(P_t)_t$ if and only if $M$ is $P_t$-compatible for all $t\in \R$. 
\end{Bem}

Now, we provide the following notion of strong causality, which, for closure procedures, is more adequate.

\begin{Def} Let $X$ be a Banach space, $M\colon D(M)\subseteqq X\to X$.

(a) Let $P\in L(X)$. We say that $M$ is \emph{norm-strongly $P$-compatible} (n-strongly $P$-compatible for short) if for all $R>0$, $x'\in X'$ the mapping
\begin{align*}
    \left(B_M(0,R),\abs{P(\cdot-\cdot)}\right)&\to \left(X,\abs{\langle P(\cdot -\cdot),x'\rangle}\right)\\
                                             f&\mapsto Mf
\end{align*}
is uniformly continuous, where $B_M(0,R)\coloneqq \{ f\in D(M); \abs{f}+\abs{Mf}< R\}$. 

(b) Let $(P_t)_t$ in $L(X)$ be a resolution of the identity. Then $M$ is called \emph{norm-strongly causal} (n-strongly causal), if $M$ is n-strongly $P_t$-compatible for all $t\in \R$.
\end{Def}

\begin{Bem}\label{rem:strong_caus_impl_causal}
    Note that if $M$ is n-strongly $P$-compatible then it is $P$-compatible. Indeed, let $f,g\in D(M)$, with $P(f-g)=0$ and $R\coloneqq \max\left\{\abs{f}+\abs{Mf},\abs{g}+\abs{Mg}\right\}+1$. By definition, for all $x'\in X'$ and $\eps>0$ there exists $\delta>0$ such that for all $f_1,f_2\in B_M(0,R)$ with $\abs{P\left(f_1-f_2\right)}<\delta$ we have $\abs{\langle P(Mf_1-Mf_2),x'\rangle}<\eps$. Thus, $\abs{\langle P(Mf-Mg),x'\rangle}<\eps$ for all $x'\in X'$ and $\eps>0$ implying \[P(Mf-Mg)=0.\]
\end{Bem}

In this section, we aim to show the following result:

\begin{Sa}\label{thm:causal_cont_causal2} Let $(X,(P_t)_t)$ be a resolution space, with $X$ reflexive. Let $M\colon D(M)\subseteqq X\to X$ linear and closable. Then the following statements are equivalent:
\begin{enumerate}[(i)]
 \item\label{caus_cont_caus12} $\overline{M}$ is causal;
 \item\label{caus_cont_caus22} $M$ is n-strongly causal.
\end{enumerate}
\end{Sa}

Regarding Remark \ref{rem:pcomp_pcausal} it suffices to establish the following:

\begin{Sa}\label{thm:causal_cont_causal} Let $X$ be a reflexive Banach space, $P\in L(X)$. Let $M\colon D(M)\subseteqq X\to X$ linear and closable. Then the following statements are equivalent:
\begin{enumerate}[(i)]
 \item\label{caus_cont_caus1} $\overline{M}$ is $P$-compatible;
 \item\label{caus_cont_caus2} $M$ is n-strongly $P$-compatible.
\end{enumerate}
\end{Sa}

For the proof of the latter theorem, we need some prerequisits.

\begin{Le}\label{le:equiv_of_caus_for_weakly-closed} Let $X$ be a reflexive Banach space, $P\in L(X)$. Let $M\colon D(M)\subseteqq X\to X$ be \emph{weakly closed}, i.e., for all $( \phi_n)_n$ in $D(M)$ we have
  \[ (\phi_n)_n, (M\phi_n)_n\text{ weakly convergent}  \Rightarrow \phi\coloneqq \textnormal{w-}\lim_{n\to\infty}\phi_n\in D(M), M\phi=\textnormal{w-}\lim_{n\to\infty} M\phi_n. 
\]
Then the following assertions are equivalent:
\begin{enumerate}[(i)]
 \item\label{equiv_cau_clo1} $M$ is $P$-compatible;
 \item\label{equiv_cau_clo2} $M$ is n-strongly $P$-compatible.
\end{enumerate} 
\end{Le}
\begin{proof}
 In Remark \ref{rem:strong_caus_impl_causal}, we have seen that \eqref{equiv_cau_clo2} implies \eqref{equiv_cau_clo1}. For the sufficiency of \eqref{equiv_cau_clo1} for \eqref{equiv_cau_clo2}, we show that $M$ is not $P$-compatible provided that $M$ is not n-strongly $P$-compatible. For this, let $R>0$, $x'\in X'$ and $\eps>0$ such that for all $n\in\N$ there are $f_n,g_n\in B_M(0,R)$ with 
\[
   \abs{P\left(f_n-g_n\right)}<\frac{1}{n}\text{ and }\abs{\langle P(Mf_n-Mg_n),x'\rangle}\geqq \eps.
\]
By boundedness of $(f_n)_n$, $(g_n)_n$, $(Mf_n)_n$ and $(Mg_n)_n$ and reflexivity of $X$, there exists a subsequence $(n_k)_k$ of $(n)_n$ such that $(f_{n_k})_k$, $(g_{n_k})_k$, $(Mf_{n_k})_k$ and $(Mg_{n_k})_k$ weakly converge. Denote the respective limits by $f,g,h_f,h_g$. With the help of the weak closedness of $M$ we deduce that $f,g\in D(M)$ and $h_f=Mf$ and $h_g=Mg$. By (weak) continuity of $P$ we get 
\[
  \abs{P(f-g)}\leqq \liminf_{k\to\infty} \abs{P(f_{n_k}-g_{n_k})}=0.
\]
 Now, from 
\[
 \abs{\langle P(Mf-Mg),x'\rangle}=\lim_{k\to\infty}\abs{\langle P(Mf_{n_k}-Mg_{n_k}),x'\rangle}\geqq \eps
\]
we read off that $M$ is not $P$-compatible.
\end{proof}

\begin{Le}\label{le:strongly_causal_is_equiv_strong_closure} Let $(X,(P_t)_t)$ be a resolution space, $M\colon D(M)\subseteqq X\to X$ closable. Then the following statements are equivalent.
\begin{enumerate}[(i)]
 \item\label{strw1} $M$ is n-strongly $P$-compatible;
 \item\label{strw2} $\overline{M}$ is n-strongly $P$-compatible.
\end{enumerate} 
\end{Le}
\begin{proof}
Let $R>0$. Then $B_M(0,R)$ is dense in $B_{\overline{M}}(0,R)$ with respect to $\abs{P(\cdot-\cdot)}$. Indeed, for $\eps>0$, $f\in B_{\overline{M}}(0,R)$ there exists $g\in B_M(0,R)$ such that 
\[
 \abs{f-g}+\abs{\overline{M}f-\overline{M}g}<\eps. 
\]
In particular, we have $\abs{P(f-g)}\leqq \Abs{P}\eps$. Assuming the validity of \eqref{strw1}, we see that
\[
   (B_{\overline{M}}(0,R),\abs{P(\cdot-\cdot)})\to (X,\abs{\langle P (\cdot-\cdot),x'\rangle}), f\mapsto \overline{M}f
\]
is uniformly continuous on a dense subset for all $x'\in X'$. This implies \eqref{strw2}. The converse is trivial.
%  It suffices to prove \eqref{strw1}$\Rightarrow$\eqref{strw2}. Let $R>0$, $a\in \R$, $x'\in X'$ and $\eps>0$. There exists $\delta>0$ such that for all $f_1,f_2\in B_M[0,R+1]$ the implication
% \[
%   \abs{\1_{\R_{<a}}(m_0)(f_1-f_2)}<\delta\Rightarrow \abs{\langle \1_{\R_{<a}}(m_0)(Mf_1-Mf_2),\phi\rangle}<\frac\eps3.
% \]
% Let $g_1,g_2 \in B_{\overline{M}}[0,R]$ with $\abs{\1_{\R_{<a}}(m_0)(g_1-g_2)}<\delta/3$. There exist $f_1,f_2\in B_{M}[0,R+1]$ such that $\max_{j\in\{1,2\}} \abs{f_j-g_j}^2+\abs{Mf_j-\overline{M}g_j}^2<\left(\min\{ \delta/3,\eps/(3\abs{\phi}+3)\}\right)^2$. Thus,
% \begin{align*}
%    \abs{\1_{\R_{<a}}(m_0)(f_1-f_2)} & \leqq \abs{\1_{\R_{<a}}(m_0)(f_1-g_1)}+\abs{\1_{\R_{<a}}(m_0)(g_1-g_2)}+\abs{\1_{\R_{<a}}(m_0)(g_2-f_2)} \\
%                                     & \leqq \abs{f_1-g_1}+\abs{\1_{\R_{<a}}(m_0)(g_1-g_2)}+\abs{g_2-f_2}<\delta. 
% \end{align*}
% Hence,
% \begin{align*}
%   &\abs{\langle \1_{\R_{<a}}(m_0)(Mg_1-Mg_2),\phi\rangle}\\ & \leqq \abs{\langle \1_{\R_{<a}}(m_0)(Mg_1-Mf_1),\phi\rangle} +\abs{\langle \1_{\R_{<a}}(m_0)(Mf_1-Mf_2),\phi\rangle}\\ &+\abs{\langle \1_{\R_{<a}}(m_0)(Mf_2-Mg_2),\phi\rangle} \\
% &\leqq \abs{Mg_1-Mf_1}\abs{\phi}+\abs{\langle \1_{\R_{<a}}(m_0)(Mf_1-Mf_2),\phi\rangle}+\abs{(Mf_2-Mg_2)}\abs{\phi}<\eps.\qedhere
% \end{align*}
\end{proof}

% \begin{Sa}\label{thm:causal_implies_closure_is_causal} Let $H$ be a Hilbert space, $\nu>0$, $M\colon D(M)\subseteqq L_\nu^2(\R;H)\to L_\nu^2(\R;H)$. Assume that $M$ is weakly closable. We denote the weak closure of $M$ by $\overline{M}^{\textnormal{w}}$. Then the following statements are equivalent:
% \begin{enumerate}[(i)]
%  \item\label{causal_implies_closure_caus1} $\overline{M}^{\textnormal{w}}$ is causal;
%  \item\label{causal_implies_closure_caus2} $M$ is strongly causal;
%  \item\label{causal_implies_closure_caus3} $\overline{M}^{\textnormal{w}}$ is strongly causal;
% \end{enumerate} 
% \end{Sa}
% \begin{proof}
%  The equivalence of \eqref{causal_implies_closure_caus1} and \eqref{causal_implies_closure_caus3} has been shown in Lemma \ref{le:equiv_of_caus_for_weakly-closed}. Moreover, the implication \eqref{causal_implies_closure_caus3}$\Rightarrow$\eqref{causal_implies_closure_caus2} is trivial. It thus remains to prove that \eqref{causal_implies_closure_caus3} is necessary for \eqref{causal_implies_closure_caus2}. This, however, is a routine argument.
% \end{proof}

\begin{proof}[Proof of Theorem \ref{thm:causal_cont_causal}]
 The assertion follows from the Lemmas \ref{le:equiv_of_caus_for_weakly-closed} and \ref{le:strongly_causal_is_equiv_strong_closure} together with the fact that for linear operators the weak closure coincides with the strong closure. Indeed, we have
\begin{multline*}
  M \text{ n-strongly $P$-compatible }\Leftrightarrow \overline{M}=\overline{M}^{\textnormal{w}} \text{ n-strongly $P$-compatible} \\ \Leftrightarrow \overline{M}^{\textnormal{w}}=\overline{M} \text{ $P$-compatible.}\qedhere
\end{multline*}
\end{proof}

\section{The non-reflexive case}\label{sec:nonrefl}

The idea to treat the non-reflexive case is to use dual pairs. We have the draw-back to only be able to treat the continuous operator case. Therefore, we allow the operator $M$ to have predomain and target spaces differing from one another. As a consequence, the notion presented becomes a bit more technical. At the end of this section, we shall sketch the connections between the notions presented. We start out with a definition.

\begin{Def} Let $X,Y$ be Banach spaces, $P\in L(X),Q\in L(Y)$. Let $X_1,Y_1$ be vector spaces and such that $\langle X,X_1\rangle$ and $\langle Y,Y_1\rangle$  become separating dual pairs. Let $M\colon D(M)\subseteqq X\to Y$. 

(a) $M$ is called \emph{$P$-$Q$-compatible}, if for all $f,g\in D(M)$ the equality $Q(M(f)-M(g))=0$ is necessary for $Pf=Pg$.

(b) $M$ is called \emph{$\sigma(X,X_1)$-$\sigma(Y,Y_1)$-strongly $P$-$Q$-compatible}, if for all $y_1\in Y_1$ the mapping
\begin{align*}
     \left(B_M(0,R),\tau_P\right)&\to \left(Y,\abs{\langle Q(\cdot-\cdot),y_1}\right) \\
                               f&\mapsto Mf, 
\end{align*}
is uniformly continuous, where $\tau_P$ is the relative topology on $B_M(0,R)$ induced by the mapping $X\ni x\mapsto Px\in (X,\sigma(X,X_1))$. Note that the latter topology is a linear topology, which in particular implies that it yields a uniform space given by the neighbourhoods of zero. 
\end{Def}

\begin{Bem} (a) If $(X,(P_t)_t)$ and $(Y,(Q_t)_t)$ are resolution spaces, then, in the above situation, we define what it means for a mapping to be causal with respect to $(Q_t)_t$-$(P_t)_t$ in a canonical way, i.e., $M$ is \emph{causal} (\emph{$\sigma(X,X_1)$-$\sigma(Y,Y_1)$-strongly causal}) if for all $t\in \R$ we have that $M$ is $P_t$-$Q_t$-compatible ($\sigma(X,X_1)$-$\sigma(Y,Y_1)$-strongly $P_t$-$Q_t$-compatible).

(b) In the previous section, for sake of presentation, we used $P=Q$ and $X=Y$, but note that the results still hold, if one replaces the target space by another resolution space $(Y,(Q_t)_t)$, with $Y$ reflexive, and define the corresponding notion of n-strong $P$-$Q$-causality. 
\end{Bem}

\begin{Bem}\label{rem:s2impl_caus} (a) Again, we verify that $P$-$Q$-compatibility is necessary for $\sigma(X,X_1)$-$\sigma(Y,Y_1)$-strongly $P$-$Q$-compatibility. For this, let $f,g\in D(M)$ with $P(f-g)=0$ and $R\coloneqq \max\{ \abs{f}+\abs{Mf},\abs{g}+\abs{Mg}\}+1$. By definition, for all $y_1\in Y_1$ and $\eps>0$ there exists a zero neighbourhood $U$ in $\tau_P$ such that if $f_1-f_2\in U$ we have $\abs{\langle Q(Mf_1-Mf_2),x_1\rangle}<\eps$. Thus, $\abs{\langle Q (Mf-Mg),y_1\rangle}<\eps$ for all $y_1\in Y_1$ and $\eps>0$, which implies $\langle Q(Mf-Mg),y_1\rangle=0$ for all $y_1\in Y_1$. Since $\langle Y,Y_1\rangle$ is separating, we deduce that $Q(Mf-Mg)=0$.

(b) Recall from Remark \ref{rem:sec1,linm} (a) that, if $M$ is linear, then $M$ is $P$-$Q$-compatible if and only if $Pf=0$ implies $QMf=0$ for all $f\in D(M)$, which in turn is equivalent to $N(P)\subseteqq N(QM)$, i.e., the nullspace of $P$ is contained in the one of $QM$.
 \end{Bem}

In this section we shall prove the following result:

\begin{Sa}\label{thm:caus_non-refl} Let $X,Y$ be Banach spaces, $P\in L(X),Q\in L(Y)$ with $P^2=P$. Let $M\colon D(M)\subseteqq X\to Y$ be densely defined, linear and continuous. Then the following statements are equivalent:
\ben
 \item\label{non-refl1} $\overline{M}$ is $P$-$Q$-compatible;
 \item\label{non-refl2} $M$ is $\sigma(X,X')$-$\sigma(Y,Y')$-strongly $P$-$Q$-compatible.
\een
\end{Sa}

In order to proceed similarly as in the previous section, we will need a little more on functional analysis, we refer to \cite{voigt2014,schaefer1971} as general references. At first, we state the following variant of Lemma \ref{le:equiv_of_caus_for_weakly-closed}.

\begin{Le}\label{lem:caus_w*_cont} Let $X,Y$ be Banach spaces, $P\in L(X),Q\in L(Y)$. Let $M\colon D(M)\subseteqq X'\to Y'$ be $\sigma(X',X)$-$\sigma(Y',Y)$ closed. Then the following statements are equivalent:
\ben
\item\label{w+1} $M$ is $P'$-$Q'$-compatible;
\item\label{w+2} $M$ is $\sigma(X',X)$-$\sigma(Y',Y)$-strongly $P'$-$Q'$-compatible.
\een 
 \end{Le}
\begin{proof} In Remark \ref{rem:s2impl_caus}, we have seen that \eqref{w+2} implies \eqref{w+1}. Now, assume that \eqref{w+2} is not true. Then there exists $R>0$, $y\in Y$, $\eps>0$ and a net of a zero neighbourhoods $(U_\alpha)_{\alpha\in I}$ in $\tau_{P'}$ with the following properties $\{ U_\alpha; \alpha\in I\}$ consitutes a zero neighbourhood basis, $\bigcap_\alpha  U_\alpha=N(P')$, $(U_\alpha)_{\alpha\in I}$ is decreasing with respect to the direction of $I$\footnote{A possible construction is to take $I\coloneqq \{ F\subseteqq X; F \text{ finite}\}$ with ``$\subseteqq$'' as partial order. For $F\in I$ let 
\[
  U_F \coloneqq \{ x'\in X'; \max_{x\in F}\abs{\langle P'x',x\rangle}\leqq 1\}.
\]
Then $\bigcap_{F\in I} U_F \supseteqq N(P')$. On the other hand, if $x'\in X'\setminus N(P')$ then there exists $x\in X$ such that $\langle x,P'x'\rangle = 2$ and, hence, $x'\notin U_{\{x\}}$.} 
 and such that for any $\alpha\in I$ there exists $f_\alpha,g_\alpha\in B_M(0,R)$ with the property that
\[
   f_\alpha-g_\alpha \in U_\alpha \text{ and }\abs{\langle Q'(Mf_\alpha-Mg_\alpha),y\rangle}\geqq \eps.
\]
   By the boundedness of $(f_\alpha)_\alpha$, $(g_\alpha)_\alpha$, $(Mf_\alpha)_\alpha$ and $(Mg_\alpha)_\alpha$ there exists a $\sigma(X'\times X'\times Y'\times Y',X\times X\times Y\times Y)$-accumulation point $(f,g,h_f,h_g)$. The closedness of $M$ implies $f,g\in D(M)$ and $Mf=h_f$ and $Mg=h_g$. Now, as $f-g \in U_\alpha$ for $\alpha$ belonging to an infinite directed subset of $I$, we deduce from $\bigcap_\alpha  U_\alpha=N(P')$ that $f-g\in N(P')$. From $\abs{\langle Q'(Mf_\alpha-Mg_\alpha),y\rangle}\geqq \eps$ for all $\alpha$ it follows that $Q'(Mf-Mg)\neq 0$, which implies that $M$ is not $P'$-$Q'$-compatible.
\end{proof}

Before we come to the proof of Theorem \ref{thm:caus_non-refl}, we recall some general Banach space theory, which might be interesting on its own right. For convenience, we state the results with the respective proofs.

\begin{Le}\label{lem:nullspace_w*_closed} Let $X,Y$ be Banach spaces, $P\in L(X,Y)$. Then $N(P'')$ is $\sigma(X'',X')$-closed. 
\end{Le}
\begin{proof}
 Let $(x''_\alpha)_\alpha$ be a net in $N(P'')$, which converges in the $\sigma(X'',X')$-topology to some $x''\in X''$. Then for any $\alpha$ and $y'\in Y'$ we have 
\[
   0=\langle P''x''_\alpha,y'\rangle=\langle x''_\alpha,P'y'\rangle \stackrel{\alpha}{\to} \langle x'',P'y'\rangle =\langle P''x'',y'\rangle. 
\]
 Thus, $P''x''=0$.
\end{proof}

\begin{Le}\label{lem:nullspace_cpred_nullspacedoubleprime} Let $X,Y$ be Banach spaces, $P\in L(X,Y)$.
 
(a) Then we have $\overline{N(P)}^{\sigma(X'',X')}\subseteqq N(P'')$.

(b) If, in addition, $X=Y$ and $P^2=P$ then $R(P')$ is $\sigma(X',X)$-closed and $\overline{N(P)}^{\sigma(X'',X')}=N(P'')$
\end{Le}
\begin{proof}
 (a) Since $P\subseteqq P''$, we have that $N(P)\subseteqq N(P'')$. Hence, the assertion follows from Lemma \ref{lem:nullspace_w*_closed}.

(b) For the first assertion, observe that with $P^2=P$, we also have $(P')^2=P'$. Take a net $(y_\alpha')_\alpha$ in $Y'$, such that $(x_\alpha')_\alpha \coloneqq (P'y_\alpha')_\alpha$ converges in $\sigma(X',X)$ to some $x'\in X'$. Hence, for all $x\in X$ and $\alpha$ we get that
\[
  \langle x',x\rangle \stackrel{\alpha}{\leftarrow}\langle P'x_\alpha,x\rangle = \langle (P')^2x_\alpha',x\rangle\\ =\langle P'x_\alpha',Px\rangle \stackrel{\alpha}{\to} \langle x',Px\rangle=\langle P'x',x\rangle.
\]
The latter implies that $x'=P'x'\in R(P')$. In order to prove the second assertion, note that in view of (a), it suffices to prove that $X''\setminus\overline{N(P)}^{\sigma(X'',X')}\subseteqq X''\setminus N(P'')$. By the Hahn-Banach theorem and the fact that $(X'',\sigma(X'',X'))'=X'$, there exists $x'\in X'$, which vanishes on $N(P)=R(P')^{\circ}$ and for which there exists $x''\in X''$ with the property that $\langle x',x''\rangle\neq 0$, where the polar ${}^\circ$ is computed with respect to the dual pair $\langle X',X\rangle$. Therefore $x'\in N(P)^{\circ}=R(P')^{\circ\circ}=\overline{R(P')}^{\sigma(X',X)}$, where the last equality follows from the bipolar theorem. Now, $R(P')$ is $\sigma(X',X)$-closed, by the first assertion of (b). Hence, $x'\in \overline{R(P')}^{\sigma(X',X)}=R(P')$. Thus, $0\neq\langle x',x''\rangle=\langle P'x',x''\rangle=\langle x',P''x''\rangle$. Hence, $x''\notin N(P'')$ as desired.
\end{proof}

\begin{Fo}\label{cor:pq-p##q##-compat} Let $X,Y$ be Banach spaces, $P\in L(X)$, $Q\in L(X,Y)$. Assume that $P^2=P$ and $N(P)\subseteqq N(Q)$. Then $N(P'')\subseteqq N(Q'')$. 
\end{Fo}
\begin{proof}
With the help of Lemma \ref{lem:nullspace_cpred_nullspacedoubleprime} (a) and (b), we deduce that
\[
  N(P'')=\overline{N(P)}^{\sigma(X'',X')}\subseteqq \overline{N(Q)}^{\sigma(X'',X')}\subseteqq N(Q'').\qedhere
\]
\end{proof}

Now, we come to the proof of Theorem \ref{thm:caus_non-refl}.

\begin{proof}[Proof of Theorem \ref{thm:caus_non-refl}] Assume $\mathcal{M}\coloneqq \overline{M}$ to be $P$-$Q$-compatible. As $\mathcal{M}$ is linear, by Remark \ref{rem:s2impl_caus}, we have that $N(P)\subseteqq N(Q\mathcal{M})$. Hence, by Corollary \ref{cor:pq-p##q##-compat}, we conclude that $N(P'')\subseteqq N((Q\mathcal{M})'')=N(Q''\mathcal{M}'')$. Therefore, $\mathcal{M}''$ is $P''$-$Q''$-compatible. As $\mathcal{M}''$ is defined on the whole of $X''$ and is $\sigma(X'',X')$-$\sigma(Y'',Y')$-continuous, we deduce with the help of Lemma \ref{lem:caus_w*_cont}, that $\mathcal{M}''$ is $\sigma(X'',X')$-$\sigma(Y'',Y')$-strongly $P''$-$Q''$-compatible. Hence, $\mathcal{M}$ is $\sigma(X'',X')$-$\sigma(Y'',Y')$-strongly $P''$-$Q''$-compatible. The assertion follows from the fact that the restriction of $\tau_{P''}$ (being defined as the initial topology induced by the mapping $X''\ni x''\mapsto P''x''\in (X'',\sigma(X'',X'))$) to $X$ coincides with $\tau_P$, the initial topology induced by $X\ni x\mapsto Px=P''x\in (X,\sigma(X,X'))$. The other implication has been proved already in Remark \ref{rem:s2impl_caus}. 
\end{proof}

A summary of the results obtained reads as follows:

\begin{Sa} Let $(X,(P_t)_t)$, $(Y,(Q_t)_t)$ be resolution spaces, with $X,Y$ reflexive. Let $M\colon D(M)\subseteqq X\to Y$ be densely defined, linear and continuous. Then the following assertions are equivalent:
\ben
  \item $\overline{M}$ is causal;
  \item $M$ is n-strongly causal;
  \item $M$ is $\sigma(X,X')$-$\sigma(Y,Y')$-strongly causal.
\een
\end{Sa}

\bibliographystyle{plain}

\end{document}